\documentclass[12pt, reqno]{amsart}
\usepackage{amsmath, amsthm, amscd, amsfonts, amssymb, mathtools}
\usepackage{color, geometry, relsize, hyperref}
\usepackage[dvipsnames]{xcolor}
\newtheorem{theorem}{Theorem}[section]
\newtheorem{lemma}[theorem]{Lemma}

\newtheorem{corollary}[theorem]{Corollary}

\theoremstyle{definition}

\newtheorem{example}[theorem]{Example}

\theoremstyle{remark}

\numberwithin{equation}{section}

\begin{document}
	
\title[Vieta-Type Formulas for Matrix Polynomials]
{Vieta-Type Formulas for Matrix Polynomials}

\author[Shrinath]{Shrinath Hadimani}
\address{School of Basic Sciences, Humanities and Management\\ 
	Manipal Institute of Technology, Manipal Academy of Higher Education\\ 
	Manipal, Karnataka-576104, India.}
\email{shrinath.hadimani@manipal.edu}

\subjclass[2010]{15A18, 15A24, 47A56}

\keywords{Matrix Polynomials; Eigenvalues of Matrix Polynomials; Matrix Vieta Formulas}	

\begin{abstract}
The classical Vieta formulas relate the coefficients of a complex scalar polynomial to the elementary 
symmetric polynomials of its roots. In this paper, we establish analogous spectral identities for complex 
matrix polynomials. For a monic matrix polynomial, we prove that the sum of all its eigenvalues equals the 
sum of the roots of the product of its diagonal scalar polynomials and is also equal to a constant multiple 
of the sum of the roots of the scalar polynomial obtained by summing its diagonal entries. We further 
show that the product of the eigenvalues of a monic matrix polynomial is determined by the determinant of its 
constant coefficient matrix. As a consequence, we recover the matrix Vieta formulas of Fuchs and Schwarz \cite{Fuchs-Schwarz}, 
for independent solutions of matrix algebraic equations via a density argument. We also derive corresponding 
identities for non-monic matrix polynomials with nonsingular leading coefficients and show, by means of an example, 
that the identities established for monic matrix polynomials do not extend directly to the non-monic case.
\end{abstract}

\maketitle

\section{Introduction}\label{sec-1}

The classical Vieta formulas establish a fundamental connection between the roots of a scalar 
polynomial and its coefficients. Let $p(\lambda)=\lambda^m+a_{m-1}\lambda^{m-1}+\cdots+a_0$ be 
a complex monic polynomial with roots $\lambda_1,\ldots,\lambda_m$, counted according to 
multiplicity. For $1\le r\le m$, the $r$-th elementary symmetric polynomial in $\lambda_1,\ldots,\lambda_m$ 
is defined by $e_r(\lambda_1,\ldots,\lambda_m)=\displaystyle \sum_{1\le i_1<\cdots< i_r\le m} \lambda_{i_1}\lambda_{i_2}\cdots\lambda_{i_r}$. 
The Vieta formulas state that $e_r(\lambda_1,\ldots,\lambda_m)=(-1)^r a_{m-r}$, for $r=1,\ldots,m$. 
Thus, the coefficients of a polynomial are completely determined by the elementary symmetric 
polynomials of its roots. In particular, for $r=1$ and $r=m$, we have the following respectively,
\begin{equation}\label{classical vieta formula}
\displaystyle \sum_{i=1}^{m}\lambda_i=-a_{m-1};
\qquad \displaystyle \prod_{i=1}^{m}\lambda_i=(-1)^m a_0.
\end{equation} These identities reveal how fundamental properties of the roots of a polynomial 
are encoded in its coefficients. 

Motivated by the classical Vieta formulas, we investigate analogous relations between the eigenvalues of a 
matrix polynomial and its coefficient matrices. For matrix polynomials, such relations are 
considerably less transparent, since the coefficients are matrices rather than scalars, and the eigenvalues 
are determined by the roots of the scalar polynomial obtained by taking the determinant of the matrix polynomial. 
Several attempts to formulate matrix analogues of the Vieta formulas have appeared in 
the literature. In \cite{Fuchs-Schwarz}, Fuchs and Schwarz considered matrix algebraic equations of the form 
$X^m+A_{m-1}X^{m-1}+\cdots+A_1X+A_0=0$, where the coefficients $A_0, A_1, \ldots, A_{m-1}$ and the 
solutions $X$ are complex square matrices of order $n$. Since the coefficients 
generally do not commute with the solutions, the classical Vieta formulas do not admit a direct extension. 
Nevertheless, for a collection of independent matrix solutions (see Section~\ref{sec-2} for the definition), 
$X_1,\ldots,X_m$, Fuchs and Schwarz \cite{Fuchs-Schwarz} established identities relating the coefficient 
matrices to the solutions. However, the resulting expressions for the coefficient matrices $A_i$'s in terms of the independent 
solutions $X_j$'s are much less elegant, than Vieta formulas. Subsequently, in \cite{Connes-Schwarz}, Connes 
and Schwarz revisited this theory, providing an alternative proof and deriving stronger identities involving 
generating functions and higher-order trace relations. These works constitute important noncommutative analogues 
of the classical Vieta theorem. More recently, in \cite{Gonzalez-Santander}, the authors proposed a heuristic method 
for solving polynomial matrix equations of the form $\displaystyle \sum_{k=0}^{m}a_kX^k =B$, where $a_k$ are complex numbers 
and $X, B$ are complex square matrices of order $n$. 

In contrast, the present work is concerned with matrix polynomials and their eigenvalues, and 
derives Vieta-type relations directly in terms of the coefficient matrices of the matrix polynomial.
In particular, we establish analogues of the identities in \eqref{classical vieta formula} and relate 
them to scalar polynomials naturally associated with the given matrix polynomial.

\section{Notations and Preliminaries}\label{sec-2}

Throughout this paper, we work over the field $\mathbb{C}$ of complex numbers. The set of all $n \times n$
matrices over $\mathbb{C}$ is denoted by $M_n(\mathbb{C})$. For a square matrix $A= \left[a_{ij}\right] \in M_n(\mathbb{C})$, 
the trace of $A$, denoted by $\operatorname{tr}(A)$, is defined as the sum of its diagonal entries, that is, 
$\operatorname{tr}(A)=\displaystyle \sum_{i=1}^{n} a_{ii}$. Recall that for matrices 
$A = \left[a_{ij}\right] \in M_m(\mathbb{C})$ and $B = \left[b_{ij}\right] \in M_n(\mathbb{C})$ their 
Kronecker product $A \otimes B$ is the $mn \times mn$ matrix defined by $A \otimes B = \begin{bmatrix}
a_{11}B & a_{12}B & \cdots & a_{1m}B \\
a_{21}B & a_{22}B & \cdots & a_{2m}B \\
\vdots & \vdots & \ddots & \vdots\\
a_{m1}B & a_{m2}B & \cdots & a_{mm}B \\
\end{bmatrix}$. Moreover, the determinant of a Kronecker product satisfies, $\det(A \otimes B) = \det(A)^n \det(B)^m$. 

An $n \times n$ matrix polynomial of degree $m$ is a function $P \colon \mathbb{C} \rightarrow M_n(\mathbb{C})$ 
defined by $P(\lambda) = \displaystyle \sum_{i=0}^{m} A_i \lambda^i$, where $A_i \in M_n(\mathbb{C})$ and $A_m \neq 0$. 
Writing $A_l=\left[
a_l^{(ij)}
\right]_{1\le i,j\le n}$, for $0 \leq l \leq m$, the matrix polynomial $P(\lambda)$ can be expressed as an $n \times n$ matrix
$P(\lambda) = \left[
p_{ij}(\lambda)
\right]$, where $p_{ij} \colon \mathbb{C} \rightarrow \mathbb{C}$ are scalar polynomials given by 
$p_{ij}(\lambda) = \displaystyle \sum_{l=0}^{m} a_l^{(ij)} \lambda^l$. Matrix polynomials have been extensively studied 
due to their numerous applications in linear systems, control theory and vibration analysis (see \cite{Kailath, Lancaster}). 
Comprehensive treatments of matrix polynomials and their spectral theory may be found in the monograph by Gohberg et al.~\cite{GLR-Book-2}.

An $n \times n$ matrix polynomial $P(\lambda)$ is said to be regular if $\det P(\lambda)$ is not identically zero as a polynomial 
in $\lambda$. For a regular matrix polynomial $P(\lambda)$, a scalar $\lambda_0 \in \mathbb{C}$ is called an eigenvalue 
if there exists a nonzero vector $u \in \mathbb{C}^n$ such that $P(\lambda_0)u = 0$. The vector $u \in \mathbb{C}^n$ is 
called an eigenvector of $P(\lambda)$ corresponding to the eigenvalue $\lambda_0$. Equivalently, $\lambda_0 \in \mathbb{C}$ is 
an eigenvalue of a regular matrix polynomial $P(\lambda)$ if and only if $\det P(\lambda_0) = 0$. Moreover, $\lambda_0 = 0$ is 
an eigenvalue of $P(\lambda)$ if and only if $A_0$ is singular. We say that $\infty$ is an eigenvalue of $P(\lambda)$ if $0$ is 
an eigenvalue of the reverse matrix polynomial $\widehat{P} (\lambda):= \lambda^m P(\frac{1}{\lambda}) = A_0 \lambda^m + A_1 \lambda^{m-1} + \cdots + A_{m-1} \lambda + A_m$. 
An $n\times n$ matrix polynomial of degree $m\ge1$ has at most $mn$ eigenvalues, counted with algebraic multiplicity. Furthermore, 
if the leading coefficient matrix is nonsingular, then the matrix polynomial has exactly $mn$ finite eigenvalues, counted with 
algebraic multiplicity.

Consider an $n\times n$ matrix polynomial $P(\lambda) = \displaystyle \sum_{i=0}^{m} A_i \lambda^i$ of degree $m$ 
with nonsingular leading coefficient matrix $A_m$. The corresponding monic matrix polynomial associated with $P(\lambda)$ 
is defined by $\widetilde P(\lambda):= I \lambda^m + V_{m-1} \lambda^{m-1} + \cdots + V_1 \lambda + V_0$, where  $V_i = A_m^{-1}A_i$, 
for $i = 0, \ldots, m-1$. Associated with $P(\lambda)$ is the $mn\times mn$ block companion matrix defined as $C_P:= \begin{bmatrix}
	0 & I & 0 & \cdots &  0\\
	0 & 0 & I & \cdots & 0 \\
	\vdots & \vdots &  \vdots & \ddots & \vdots\\
	0 & 0 & 0 & \cdots & I\\
	-V_0 & -V_1 & -V_2 & \cdots & -V_{m-1}
\end{bmatrix}$. The matrix polynomials $P(\lambda)$, $\widetilde P(\lambda)$, together with the associated block 
companion matrix $C_P$, have the same eigenvalues (see \cite{Higham-Tisseur} for details). We state this as Lemma below 
for convenience.

\begin{lemma}[\cite{Higham-Tisseur}]\label{lem-1}
If $P(\lambda)$ is a complex matrix polynomial with nonsingular leading coefficient,
then the eigenvalues of $P(\lambda)$, $\widetilde P(\lambda)$ and $C_P$ are the same. Further, $u \in \mathbb{C}^n$ is an
eigenvector corresponding to an eigenvalue $\lambda_0$ of $P(\lambda)$ if and only if $u \in \mathbb{C}^n$ is an 
eigenvector corresponding to the eigenvalue $\lambda_0$ of $\widetilde P(\lambda)$ if and only if $\begin{bmatrix}
	u \\ \lambda_0 u \\ \vdots \\ \lambda_0^{m-1} u
\end{bmatrix} \in \mathbb{C}^{mn}$  is an eigenvector
corresponding the eigenvalue $\lambda_0$ of $C_P$ .
\end{lemma}

The matrices $X_1, X_2, \ldots, X_m \in M_n(\mathbb{C})$ are said to be independent if the determinant of the block 
Vandermonde matrix $V = \begin{bmatrix}
I & I & \dots & I \\
X_1 & X_2 & \dots & X_m \\
\vdots & \vdots & \ddots & \vdots \\
X_1^{m-1} & X_2^{m-1} & \dots & X_m^{m-1} \\
\end{bmatrix}$ is nonzero. Consider the matrix algebraic equation, \begin{equation}\label{eq-matrix algebraic equation}
X^m+A_{m-1}X^{m-1}+\cdots+A_1X+A_0=0,
\end{equation}
with the coefficients $A_0, A_1, \ldots, A_{m-1}$ and the solutions $X$ in $M_n(\mathbb{C})$. A fundamental connection 
between solutions of \eqref{eq-matrix algebraic equation} and the associated matrix polynomial 
$P(\lambda) = I \lambda^m+A_{m-1}\lambda^{m-1}+\cdots+A_1\lambda+A_0$ is that every eigenvalue of a solution 
$X \in M_n(\mathbb{C})$ is also an eigenvalue of $P(\lambda)$. Using this fact, in \cite{Fuchs-Schwarz}, authors proved that 
if $X_1, X_2, \ldots, X_m \in M_n(\mathbb{C})$ are independent solutions of Equation \eqref{eq-matrix algebraic equation}, then
\begin{equation}\label{eq-matrx vieta theorem}
	\operatorname{tr}(A_{m-1})=-\displaystyle \sum_{i=1}^{m}\operatorname{tr}(X_i);
	\qquad \det(A_0)=(-1)^{nm}\displaystyle \prod_{i=1}^{m}\det(X_i).
\end{equation} The identities established in Theorems \ref{thm-matrix vieta thm-1} and \ref{thm-matrix vieta thm-2} may be viewed as spectral 
analogues of the matrix Vieta formulas given in Equation \eqref{eq-matrx vieta theorem}. In fact, we show that the identities 
in Equation \eqref{eq-matrx vieta theorem} can be derived from these spectral identities. More precisely, 
by combining Theorems \ref{eq-thm-1-1} and \ref{eq-thm-1-3} with a density argument analogous to that used in \cite{Fuchs-Schwarz}, 
we recover the identities in Equation \eqref{eq-matrx vieta theorem} for all independent solutions of 
Equation \eqref{eq-matrix algebraic equation}.

Throughout this paper, we assume that the leading coefficient matrix of every matrix polynomial are nonsingular.
 
\section{Main Results}\label{sec-3}	
This section contains the main results of the manuscript. The section is further divide for ease of reading.

\subsection{Vieta-Type Formulas for Monic Matrix Polynomials}\label{sec-3.1}

In this section, we derive Vieta-type identities for monic matrix polynomials. We begin with the following theorem, 
which establishes a relation between the eigenvalues of a monic matrix polynomial and the roots of the product of 
its diagonal entries. In the course of the proof, we show that the sum of the eigenvalues of the matrix polynomial 
is equal to the negative of the trace of the coefficient of the second-highest degree term.

\begin{theorem}\label{thm-matrix vieta thm-1}
Let $P(\lambda) = I \lambda^m + A_{m-1}\lambda^{m-1} + \dots + A_1\lambda + A_0$ be an $n \times n$ monic matrix 
polynomial, where $A_l \in M_n(\mathbb{C})$ for $0 \leq l \leq m-1$. Then the sum of all eigenvalues of $P(\lambda)$ 
equals the sum of the roots of the scalar polynomial $\displaystyle \prod_{i=1}^{n} p_{ii}(\lambda)$.
\end{theorem}

\begin{proof}
Let $\lambda_1, \lambda_2, \ldots, \lambda_{mn}$ denote the eigenvalues of the monic matrix polynomial $P(\lambda)$, 
counted according to algebraic multiplicity. 
By Lemma~\ref{lem-1}, the eigenvalues of $P(\lambda)$ coincide with those of its block companion matrix
$C_P =\begin{bmatrix}
			0 & I & 0 & \cdots & 0 \\
			0 & 0 & I & \cdots & 0 \\
			\vdots & \vdots & \vdots & \ddots & \vdots \\
			0 & 0 & 0 & \cdots & I \\
			-A_0 & -A_1 & -A_2 & \cdots & -A_{m-1}
\end{bmatrix}$. 
Hence, 
\begin{equation}\label{eq-thm-1-1}
\sum_{k=1}^{mn}\lambda_k = \operatorname{tr}(C_P).
\end{equation} 
Since all diagonal blocks of $C_P$, except the last are zero, its trace is
\begin{equation}\label{eq-thm-1-2}
\operatorname{tr}(C_P) = -\operatorname{tr}(A_{m-1}).
\end{equation}
Write $A_l = \left[a_l^{(ij)}\right]_{1 \le i,j \le n}$, for $0 \leq l \leq m-1$ and 
$P(\lambda) = \left[p_{ij}(\lambda) \right]_{1 \le i,j \le n}$, where
$p_{ij}(\lambda) = \displaystyle \sum_{l=0}^{m} a_l^{(ij)} \lambda^l$. Since $P(\lambda)$ is monic, 
each diagonal entry is a monic scalar polynomial of the form
$p_{ii}(\lambda) = \lambda^m + a_{m-1}^{(ii)} \lambda^{m-1} + \dots + a_0^{(ii)}$, $1 \leq i \leq n$. 
Hence, the sum of the roots of $p_{ii}(\lambda)$ 
is $-a_{m-1}^{(ii)}$.
Therefore, 
\begin{equation}\label{eq-thm-1-3}
\text{sum of the roots of}
\displaystyle \prod_{i=1}^n p_{ii}(\lambda) =
\displaystyle \sum_{i=1}^n \left(-a_{m-1}^{(ii)}\right) = -\operatorname{tr}(A_{m-1}).
\end{equation}
Combining Equations \eqref{eq-thm-1-1}, \eqref{eq-thm-1-2} and \eqref{eq-thm-1-3} yields 
\begin{center}
$\displaystyle \sum_{k=1}^{mn}\lambda_k = \text{sum of the roots of} \displaystyle \prod_{i=1}^n p_{ii}(\lambda)$,
\end{center}
which completes the proof.
\end{proof}
	
The following corollary is consequence of Theorem \ref{thm-matrix vieta thm-1} and relates the sum of the eigenvalues 
of a monic matrix polynomial to the sum of the roots of a scalar polynomial obtained from its diagonal entries.

\begin{corollary}\label{cor-matrix vieta thm-1}
Let $P(\lambda)=I\lambda^m+A_{m-1}\lambda^{m-1}+\cdots+A_0$ be an $n \times n$ monic matrix polynomial, and 
let $S(\lambda)= \displaystyle \sum_{i=1}^{n}p_{ii}(\lambda)$, where $p_{ii}(\lambda)$ denotes the $i$th diagonal 
entry of $P(\lambda)$. If $\lambda_1,\ldots,\lambda_{mn}$ are the eigenvalues of $P(\lambda)$, counted with algebraic 
multiplicity, then $\displaystyle \sum_{k=1}^{mn}\lambda_k = n\cdot\left(\text{sum of the roots of }S(\lambda)\right)$,
where the roots of $S(\lambda)$ are counted with multiplicity.	
\end{corollary}

\begin{proof}
Since $S(\lambda)=\displaystyle \sum_{i=1}^{n}p_{ii}(\lambda)$, and each diagonal entry $p_{ii}(\lambda)$ is 
monic of degree $m$, we have $S(\lambda)= \displaystyle \sum_{i=1}^n p_{ii}(\lambda)
= n\lambda^m + \left(\sum_{i=1}^n a_{m-1}^{(ii)}\right)\lambda^{m-1} + \cdots + \left(\sum_{i=1}^n a_0^{(ii)}\right)$.
Thus, the leading coefficient of $S(\lambda)$ is $n$, and the coefficient of $\lambda^{m-1}$ is $\displaystyle \sum_{i=1}^n a_{m-1}^{(ii)}=\operatorname{tr}(A_{m-1})$.
Hence, by Vieta's formula for scalar polynomials,
\begin{equation}\label{eq:sum-roots-S}
\text{sum of roots of }S(\lambda)= -\frac{\operatorname{tr}(A_{m-1})}{n}.
\end{equation}
On the other hand, by Theorem \ref{thm-matrix vieta thm-1},
\begin{equation}\label{eq:sum-eigs-P}
		\sum_{k=1}^{mn}\lambda_k = -\operatorname{tr}(A_{m-1}).
\end{equation}
Combining \eqref{eq:sum-roots-S} and \eqref{eq:sum-eigs-P} yields $\displaystyle \sum_{k=1}^{mn}\lambda_k = n\cdot\big(\text{sum of the roots of }S(\lambda)\big)$. 
This completes the proof.
\end{proof}	
	
In the scalar setting, Vieta's formulas express the product of the roots of a polynomial in terms of its 
constant coefficient. The following theorem establishes an analogous identity for matrix polynomials.	
	
\begin{theorem}\label{thm-matrix vieta thm-2} 
Let $P(\lambda) = I \lambda^m + A_{m-1}\lambda^{m-1} + \dots + A_1\lambda + A_0$ be an $n \times n$ monic matrix 
polynomial, where $A_l \in M_n(\mathbb{C})$ for $0 \le l \le m-1$. If $\lambda_1,\ldots,\lambda_{mn}$ are the 
eigenvalues of $P(\lambda)$, counted with algebraic multiplicity, then 
$\displaystyle \prod_{k=1}^{mn}\lambda_k = (-1)^{mn}\det(A_0)$.
\end{theorem}

\begin{proof}Consider the block matrix $\widetilde{I} =
\begin{bmatrix}
    	0 & 0 & 0 & \cdots & 0 & I  \\
    	I & 0 & 0 & \cdots & 0 & 0  \\
    	0 & I & 0 & \cdots & 0 & 0  \\
    	\vdots & \vdots & \vdots & \ddots & \vdots \\
    	0 & 0 & 0 & \cdots & I & 0 
\end{bmatrix}$ of size $mn \times mn$, where each block is of size $n \times n$. Pre-multiplying $C_P$ by 
$\widetilde{I}$ yields the matrix $\widetilde{C_P} =
\begin{bmatrix}
    -A_0 & -A_1 & -A_2 & \cdots & -A_{m-1} \\
    0 & I & 0 & \cdots & 0 \\
    0 & 0 & I & \cdots & 0 \\
    \vdots & \vdots & \vdots & \ddots & \vdots \\
    0 & 0 & 0 & \cdots & I
\end{bmatrix}$. 
Using the Kronecker product, we can express $\widetilde{I} = \Pi \otimes I$, where $\Pi=\begin{bmatrix}
    0 & 0 & 0 &\cdots & 0 & 1 \\
    1 & 0 & 0 &\cdots & 0 & 0 \\
    0 & 1 & 0 & \cdots & 0 & 0 \\
    \vdots & \vdots & \vdots & \ddots & \vdots & \vdots \\
    0 & 0 & 0 & \cdots & 1 & 0
\end{bmatrix}$ is an $m \times m$ permutation matrix and $I$ is the identity matrix of order $n$. 
Hence, by the determinant formula for the Kronecker product,
    \begin{equation}\label{eq-det-1}
    	\det(\widetilde{I})= \det(\Pi \otimes I) = \det(\Pi)^n \det(I)^m = \det(\Pi)^n.
    \end{equation}     
Since $\Pi$ corresponds to the cyclic permutation $(1\,2\,\ldots\,m)$, its determinant is $\det(\Pi) = (-1)^{m-1}$. 
Therefore, from Equation \eqref{eq-det-1},
\begin{equation}\label{eq-det-2}
	\det(\widetilde{C_P}) = \det(\widetilde{I} C_P) = \det(\widetilde{I}) \det(C_P) = (-1)^{n(m-1)} \det(C_P). 
\end{equation}   
Since $\widetilde{C_P}$ is block upper triangular, we have
\begin{equation}\label{eq-det-3}
	\det(\widetilde{C_P}) = \det(-A_0) \det(I)^{m-1} = (-1)^n \det(A_0).
\end{equation}
Combining Equations $\eqref{eq-det-2}$ and $\eqref{eq-det-3}$, we obtain
\begin{equation}
	\det(C_P) = (-1)^{(m-1)n} (-1)^n \det(A_0) = (-1)^{mn}\det(A_0).
\end{equation}\label{eq-det-4}
By Lemma \ref{lem-1}, the eigenvalues of $P(\lambda)$ coincide with those of its block companion matrix
$C_P$. Hence, $\displaystyle \prod_{k=1}^{mn}\lambda_k =\det(C_P)$. Therefore, from Equation \eqref{eq-det-4}, 
$\displaystyle \prod_{k=1}^{mn}\lambda_k=(-1)^{mn}\det(A_0)$, which completes the proof.
\end{proof}

We now derive the matrix Vieta formulas established in Theorem $1.1$ of \cite{Fuchs-Schwarz} as a consequence of 
Theorems \ref{thm-matrix vieta thm-1} and \ref{thm-matrix vieta thm-2}.

\begin{theorem}
Let $X_1,X_2,\ldots,X_m \in M_n(\mathbb C)$ be independent solutions of Equation \eqref{eq-matrix algebraic equation}. 
Then $\operatorname{tr}(A_{m-1})= \displaystyle -\sum_{i=1}^{m}\operatorname{tr}(X_i)$ and
$\det(A_0)=(-1)^{mn} \displaystyle \prod_{i=1}^{m}\det(X_i)$.
\end{theorem}

\begin{proof}
Let $U$ denote the set of all $m$-tuples $(X_1,\ldots,X_m)$ of independent solutions of Equation \eqref{eq-matrix algebraic equation},
and let $V$ denote the set of all $m$-tuples $(X_1,\ldots,X_m)$ of solutions of Equation \eqref{eq-matrix algebraic equation} 
whose eigenvalues consist of $mn$ pairwise distinct complex numbers. By \cite{Fuchs-Schwarz}, the sets $U$ and $V$ 
are nonempty Zariski-open subsets of $M_n(\mathbb C) \times \cdots \times  M_n(\mathbb C)$. Hence $U\cap V$ is dense in $U$.
Let $(X_1,\ldots,X_m)\in U\cap V$. Since every eigenvalue of a solution of the matrix equation \eqref{eq-matrix algebraic equation} 
is an eigenvalue of the associated matrix polynomial $P(\lambda)=I\lambda^m+A_{m-1}\lambda^{m-1}+\cdots+A_1\lambda+A_0$,
all eigenvalues of $X_1,\ldots,X_m$ are eigenvalues of $P(\lambda)$. Moreover, because the eigenvalues of $X_1,\ldots,X_m$ are 
pairwise distinct and together comprise $mn$ eigenvalues, they constitute the complete set of eigenvalues of $P(\lambda)$. 
Hence, if $\lambda_1, \ldots, \lambda_{mn}$ are the eigenvalues of $P(\lambda)$, then 
$\displaystyle \sum_{k=1}^{mn}\lambda_k =  \sum_{i=1}^{m}\operatorname{tr}(X_i)$. Applying Theorem $\ref{thm-matrix vieta thm-1}$ 
we obtain $\operatorname{tr}(A_{m-1})= \displaystyle -\sum_{i=1}^{m}\operatorname{tr}(X_i)$. Similarly, since 
$\displaystyle \prod_{i=1}^{m}\det(X_i) = \prod_{i=1}^{mn} \lambda_i$, by Theorem $\ref{thm-matrix vieta thm-2}$ 
we obtain $\det(A_0)= \displaystyle (-1)^{mn}\prod_{i=1}^{m}\det(X_i)$. Thus, both identities hold on the dense subset $U\cap V$ of $U$.
	
Now define $F(X_1,\ldots,X_m)=\operatorname{tr}(A_{m-1})+\displaystyle \sum_{i=1}^{m}\operatorname{tr}(X_i)$
and $G(X_1,\ldots,X_m)=\det(A_0)-(-1)^{mn}\prod_{i=1}^{m}\det(X_i)$.
Both $F$ and $G$ are polynomial, and hence continuous functions of the entries of $X_1,\ldots,X_m$.
Since $F=0$ and $G=0$ on the dense subset $U\cap V$ of $U$, it follows that $F=0$ and $G=0$ on all of $U$.
Consequently, $\operatorname{tr}(A_{m-1})= \displaystyle-\sum_{i=1}^{m}\operatorname{tr}(X_i)$ and
$\det(A_0)= \displaystyle (-1)^{mn}\prod_{i=1}^{m}\det(X_i)$ for every independent $m$-tuple of solutions.
\end{proof}

\subsection{Vieta-Type Formulas for Non-monic Matrix Polynomials with nonsingular leading coefficient}\label{sec-3.2}

In this section, we derive Vieta-type identities for non-monic matrix polynomials  with nonsingular leading coefficients. 
Unlike the identities established in Theorems \ref{thm-matrix vieta thm-1}, \ref{thm-matrix vieta thm-2}, and 
Corollary \ref{cor-matrix vieta thm-1} for monic matrix polynomials, the identities obtained here depend explicitly 
on the leading coefficient matrix. We first 
establish the main result and then illustrate this distinction with an example.

\begin{theorem}
Let $P(\lambda)=L\lambda^m+A_{m-1}\lambda^{m-1}+\cdots+A_0$ be an $n \times n$ matrix polynomial with nonsingular leading 
coefficient matrix $L$, and let $S(\lambda)= \displaystyle \sum_{i=1}^{n}p_{ii}(\lambda)$, where $p_{ii}(\lambda)$ denotes 
the $i$th diagonal entry of $P(\lambda)$. If $\lambda_1,\ldots,\lambda_{mn}$ are the eigenvalues of $P(\lambda)$, counted 
with algebraic multiplicity, then
\begin{enumerate}
	\item $\displaystyle \sum_{k=1}^{mn}\lambda_k = -\operatorname{tr} \left( L^{-1}A_{m-1} \right)$,
	
	\item $\displaystyle \prod_{k=1}^{mn}\lambda_k = (-1)^{mn}\det(L^{-1} A_0)$.
	
	\item the sum of the roots of $S(\lambda)$ is $\displaystyle -\frac{\operatorname{tr}(A_{m-1})}{\operatorname{tr}L}$, 
	provided $\operatorname{tr}L \neq 0$.
\end{enumerate}
\end{theorem}

\begin{proof}
Since $L$ is nonsingular matrix, consider the corresponding monic matrix polynomial
$\widetilde P(\lambda) := L^{-1}P(\lambda)= I\lambda^m + (L^{-1}A_{m-1})\lambda^{m-1} + \cdots+ (L^{-1}A_0)$ 
and let $C_P$ denote its associated block companion matrix. By Lemma \ref{lem-1}, the eigenvalues of 
$P(\lambda)$, $\widetilde P(\lambda)$ and $C_P$ are the same. Using the same trace argument as in the proof 
of Theorem \ref{thm-matrix vieta thm-1}, we obtain $\displaystyle \sum_{k=1}^{mn}\lambda_k = -\operatorname{tr}(L^{-1}A_{m-1})$. 
Furthermore, applying Theorem \ref{thm-matrix vieta thm-2} to the monic matrix polynomial $\widetilde P(\lambda)$ yields 
$\displaystyle \prod_{k=1}^{mn}\lambda_k = (-1)^{mn}\det(L^{-1} A_0)$. Next, by construction the scalar polynomial 
$S(\lambda)=\displaystyle \sum_{i=1}^n p_{ii}(\lambda) = (\operatorname{tr}L)\lambda^m + \operatorname{tr}(A_{m-1})\lambda^{m-1} + \cdots +\operatorname{tr}(A_0)$. 
Thus, the leading coefficient of $S(\lambda)$ is $\operatorname{tr}(L)$, and the coefficient of $\lambda^{m-1}$ 
is $\operatorname{tr}(A_{m-1})$. By Vieta's formula for scalar polynomials, the sum of roots of 
$S(\lambda) = \displaystyle -\frac{\operatorname{tr}(A_{m-1})}{\operatorname{tr}L}$. This completes the proof.

\end{proof}

We conclude this section with an example illustrating that the identities established in
Theorems \ref{thm-matrix vieta thm-1}, \ref{thm-matrix vieta thm-2}, and Corollary \ref{cor-matrix vieta thm-1}
for monic matrix polynomials do not, in general,
extend to non-monic matrix polynomials with nonsingular leading coefficients.
\begin{example}
Consider the linear matrix polynomial
\begin{center}
	$P(\lambda)=\begin{bmatrix}
		2 & 1 \\ 
		1 & 2
	\end{bmatrix} \lambda + \begin{bmatrix}
		1 & -1 \\ 
		1 & 1
	\end{bmatrix} = \begin{bmatrix}
		 2 \lambda + 1 & \lambda -1 \\
		\lambda +1 & 2 \lambda +1
	\end{bmatrix}$.
\end{center} 

The eigenvalues of $P(\lambda)$ are the zeros of $\det P(\lambda) = 3\lambda^2 +4 \lambda +2$ namely, 
$\lambda_1 =\frac{-2+i\sqrt{2}}{3}$ and $\lambda_2=\frac{-2-i\sqrt{2}}{3}$. Hence, $\lambda_1 + \lambda_2 = -\frac{4}{3}$ 
and $\lambda_1 \lambda_2 = \frac{2}{3}$. The diagonal entries of $P(\lambda)$ are $p_{11}(\lambda)=p_{22}(\lambda)=2\lambda + 1$. 
Therefore, $p(\lambda)=p_{11}(\lambda)p_{22}(\lambda)= (2\lambda + 1)^2$, whose only root is $-\frac{1}{2}$ with multiplicity 
two. Thus the sum of its roots is $-1$. Moreover, $S(\lambda)=p_{11}(\lambda)+p_{22}(\lambda)=2(2\lambda + 1)$, whose unique 
root is $-\frac{1}{2}$. Consequently, $\displaystyle \sum_{k=1}^{2}\lambda_k =-\frac{4}{3}
\neq -1 = \text{sum of the roots of} \ p(\lambda)$, showing that the conclusion of Theorem \ref{thm-matrix vieta thm-1} 
does not hold for non-monic matrix polynomials. Furthermore,
$2\left(\text{sum of the roots of }S(\lambda)\right) = 2\left(-\frac{1}{2}\right) = -1 \neq -\frac{4}{3} =\displaystyle  \sum_{k=1}^{2}\lambda_k$, 
so the conclusion of Corollary \ref{cor-matrix vieta thm-1} also fails for non-monic matrix polynomials.
Finally, $\displaystyle  \prod_{k=1}^{2}\lambda_k = \frac{2}{3} \neq 2 = \det(A_0)$, where $A_0=
\begin{bmatrix}
	1 & -1\\
	1 & 1
\end{bmatrix}$.
Hence, the conclusion of Theorem
\ref{thm-matrix vieta thm-2} does not extend to non-monic matrix polynomials with nonsingular leading coefficients.

\end{example}

%\medskip
%\noindent
%{\bf Acknowledgements:} Pallavi .B and Shrinath Hadimani acknowledge the Council of 
%Scientific and Industrial Research (CSIR) and the University Grants Commission (UGC), 
%Government of India, for financial support through research fellowships.

\bibliographystyle{amsplain}

\end{document}